\documentclass{amsart}
\usepackage{amscd,amssymb, amsfonts, amsmath}
\usepackage{mathrsfs} 

\newtheorem{theorem}{Theorem}[section]
\newtheorem{proposition}[theorem]{Proposition}

\newtheorem{corollary}[theorem]{Corollary}
\theoremstyle{definition}
\newtheorem{definition}[theorem]{Definition}

\theoremstyle{remark}

\numberwithin{equation}{section}

\usepackage{xy}
\input xy
\xyoption{all}
\newcommand{\R}{\mathbb{R}}

\DeclareMathOperator{\dom}{dom}

\DeclareMathOperator{\colim}{colim}

\renewcommand{\hom}{\operatorname{Hom}}

\DeclareMathOperator{\Sing}{Sing}

\newcommand{\po}{\ar@{}[dr]|(.7){\Searrow}}
\newcommand{\pb}{\ar@{}[dr]|(.3){\Nwarrow}}

\DeclareMathOperator{\Map}{Map}
\DeclareMathOperator{\ev}{Ev}
\DeclareMathOperator{\codom}{codom}
\newcommand{\Top}{\operatorname{Top}}

\newcommand{\mathcolon}{\colon\,}
\renewcommand{\smash}{\wedge}
\newcommand{\cat}[1]{\mathcal{#1}}

\newcommand{\spectra}{\text{Sp}^{O}}
\newcommand{\Gspectra}{G\text{-}\spectra }
\newcommand{\boxprod}{\mathbin\square}

\begin{document}

\title{An alternative approach to equivariant stable homotopy theory} 


\author{Mark Hovey}
\address{Department of Mathematics \\ Wesleyan University
\\ Middletown, CT 06459}
\email{hovey@member.ams.org}

\author{David White}
\address{Department of Mathematics \\ Wesleyan University
\\ Middletown, CT 06459}
\email{dwhite03@wesleyan.edu}


\begin{abstract} 
Building on the work of Martin Stolz~\cite{stolz-thesis}, we develop
the basics of equivariant stable homotopy theory starting from the
simple idea that a $G$-spectrum should just be a spectrum with an
action of $G$ on it, in contrast to the usual approach in which the
definition of a $G$-spectrum depends on a choice of universe.
\end{abstract}

\maketitle

\section{Introduction}

In the standard approach to equivariant stable homotopy theory,
pioneered by Lewis, May, and Steinberger
in~\cite{lewis-may-steinberger} and reaching the current state of the
art in the work of Mandell and
May~\cite{mandell-may-equivariant-orthogonal}, the definition of a
$G$-spectrum for a compact Lie group $G$ depends on a choice of
$G$-universe $\cat{U}$; that is, on a set of orthogonal
$G$-representations closed under finite direct sums and summands,
containing the trivial one-dimensional representation. We convert the
universe into a category enriched over pointed $G$-spaces by defining
$\cat{U} (V,W)$ to be the space of orthogonal isomorphisms $O
(V,W)_{+}$ with $G$ acting by conjugation.  A $\cat{U}$-space is then
an enriched functor from $X\mathcolon \cat{U}\xrightarrow{}G\Top_{*}$
to the category of based $G$-spaces.  The sphere $S$ is a
$\cat{U}$-space, where $S (V)=S^{V}$, the one-point compactification
of $V$.  A $G\cat{U}$-spectrum $X$ is then an external module over
$S$, in the sense that we have an associative and unital natural
transformation of functors $\cat{U}\times
\cat{U}\xrightarrow{}G\Top_{*}$ from $S (V)\wedge X (W)$ to $X
(V\oplus W)$.  Then one defines homotopy groups $\pi_{q}^{H}(X)$ for
each closed subgroup $H$ of $G$ and for all integers $q$, and declares
the homotopy isomorphisms to be the weak equivalences.

Historically, the universe has been thought to be central to the
definition of a $G$-spectrum.  For example, a $G\cat{U}$-spectrum
where $\cat{U}$ consists of the trivial $G$-representations has been
called a ``naive'' $G$-spectrum, whereas a $G\cat{U}$-spectrum based
on a complete $G$-universe $\cat{U}$, in which every
finite-dimensional orthogonal $G$-representation occurs, has been
called a ``genuine'' $G$-spectrum. 

In this paper we show that we can think of a universe as a Quillen
model structure on the category of naive $G$-spectra.  That is:
\begin{enumerate}
\item For us, a $G$-spectrum is always an orthogonal spectrum $X$ with a
continuous action of $G$ on it, so a set of based $G\times O
(n)$-spaces $X_{n}$ with associative unital $G\times O (p)\times O
(q)$-structure maps 
\[
S^{p}\wedge X_{q} \xrightarrow{} X_{p+q}.
\]
\item Given a $G$-spectrum $X$ and an orthogonal $G$-representation
$V$ of dimension $n$, there is a $G$-space $\ev_{V} (X)= X
(V)$ defined by $X (V)=X_{n}$ as an $O (n)=O (V)$-space, with $G$-action
\[
g\cdot x=\rho (g) (gx) = g (\rho (g)x))
\]
where $\rho \mathcolon G\xrightarrow{}O (n)$ corresponds to $V$. The
functor $\ev_{V}$ from $G$-spectra to based $G$-spaces has a left
adjoint $F_{V}$.  The structure maps of $X$ then make $X$ into a
$G\cat{U}$-spectrum for any universe $\cat{U}$.  Thus the category of
naive $G$-spectra is equivalent to the category of $G\cat{U}$-spectra
for any $G$-universe $\cat{U}$.
\item Given a $G$-universe $\cat{U}$, there is a symmetric monoidal
$\cat{U}$-level model structure on $G$-spectra, in which a map $f$ is
a weak equivalence or fibration if and only if $f (V)$ is a weak
equivalence or fibration of $G$-spaces for all $V\in \cat{U}$.
\item One can then form the smallest symmetric monoidal Bousfield
localization of the $\cat{U}$-level model structure with respect to
the maps
\[
S^{V} \wedge F_{V}S^{0} = F_{V}S^{V} \xrightarrow{} S=F_{0}S^{0}
\]
for $V\in \cat{U}$, where this map is adjoint to the identity map
$S^{V}\xrightarrow{}\ev_{V}F_{0}S^{0}=S^{V}$.  In the homotopy
category of this Bousfield localization, smashing with $S^{V}$ for
$V\in \cat{U}$ is an equivalence with inverse given by smashing with
$F_{V}S^{0}$.  The resulting model category structure coincides with
the stable model structure on $G\cat{U}$-spectra
of~\cite{mandell-may-equivariant-orthogonal} under the equivalence
between $G$-spectra and $G\cat{U}$-spectra.
\end{enumerate}

We stress that the general idea of this paper has been known to the
experts for a long time. In Elmendorf and May's 1997
paper~\cite{elmendorf-may}, they showed that equivariant $S$-modules
are independent of the universe up to equivalence and that different
universes correspond to different model category structures.  In
Section~V.I of~\cite{mandell-may-equivariant-orthogonal}, Mandell and
May show if $\cat{U}'$ is a subuniverse of $\cat{U}$, then the
category of $G\cat{U}$-spectra is equivalent to the category of
$G\cat{U}'$-spectra, and they describe the model structure on
$G\cat{U}$-spectra for a complete universe $\cat{U}$ that corresponds
to a subuniverse $\cat{U}'$.  Our approach is the reverse to theirs,
as we start with the trivial universe; we acknowledge a debt to Neil
Strickland, who told the first author such an approach was possible
sometime around the year 2000.  Stolz's 2011
thesis~\cite{stolz-thesis} contains a very similar approach to ours,
the primary difference being that we use Bousfield localization
instead of stable homotopy isomorphisms, and the secondary difference
being that Stolz prefers to use the more abstruse definition of a
$G$-orthogonal spectrum as a certain kind of $G$-functor.  It would be
fair to think of this paper as a popularization and simplification of
Stolz's work.

\section{$G$-spaces and $G$-spectra}

We first fix notation.  A topological space is a compactly generated,
weak Hausdorff space, and all constructions, such as limits and
colimits, are carried out in this bicomplete closed symmetric monoidal
category $\Top$ or its pointed analogue $\Top_{*}$. The symbol $G$
will always denote a compact Lie group.  A $G$-space is a space with a
continuous left action of $G$, and a based $G$-space is a $G$-space
with a distinguished basepoint that is fixed by the action of $G$.
The category $\Top_{G}$ of $G$-spaces and nonequivariant maps is 
closed symmetric monoidal, where we use the diagonal action of $G$ on
$X\times Y$, and the conjugation action of $G$ on the (non-equivariant)
mapping space $\Map (X,Y)$.  That is, $(g\cdot f)(x)=g\cdot
f(g^{-1}\cdot x)$.  The category $G\Top$ of $G$-spaces and equivariant
maps is also closed symmetric monoidal, as a subcategory of $G$-spaces
and nonequivariant maps.  The category $G\Top$ is enriched, tensored,
and cotensored over $\Top$ via the symmetric monoidal left adjoint
that takes $X\in \Top$ to $X$ with the trivial $G$-action.  The
enrichment over $\Top$ is then given by the subspace $\Map_{G} (X,Y)$
of equivariant maps.  

Given a $G$-space $X$ and a closed subgroup $H$ of $G$, we can
consider the fixed points $X^{H}$.  This is a functor from $G\Top$ to
$\Top$ with left adjoint the functor that takes $Y$ to $G/H \times Y$.
(Note that the reason to assume that $H$ is closed is so that the
usual definition of $G/H$ as cosets of $H$ with the quotient topology
is weak Hausdorff---if $H$ were not closed then we would have to take
the closure anyway to stay in weak Hausdorff spaces).  

The category $G\Top$ and its pointed analogue $G\Top_{*}$ are proper,
cellular, topological, symmetric monoidal model categories.  For the standard
notions of model category theory, see~\cite{hovey-model}
or~\cite{hirschhorn}.  A map $f$ in $G\Top_{*}$ is a weak equivalence
(resp. fibration) if and only if $f^{H}$ is a weak equivalence
(resp. fibration) in $\Top_{*}$ for every closed subgroup $H$.  The
generating cofibrations are the maps 
\[
(G/H)_{+} \wedge S^{n-1}_{+} \xrightarrow{} (G/H)_{+} \wedge D^{n}_{+}
\]
for all $n\geq 0$ (where $S^{-1}$ is the empty set) and for all closed
subgroups $H$ of $G$.  The generating trivial cofibrations are the maps
\[
(G/H)_{+} \wedge D^{n}_{+} \xrightarrow{} (G/H)_{+} \wedge D^{n}_{+}
\wedge D^{1}_{+}
\]
for $n\geq 0$ and for all closed subgroups $H$ of $G$.  

We note that there are in fact many different model structures on
$G\Top_{*}$, one for each collection of closed subgroups of $G$, and
also that $G$ need not be compact Lie for this to work.  However,
there are some subtleties with the symmetric monoidal structure when
one works with a general collection of closed subgroups.
Fausk~\cite{fausk} has an excellent treatment of these issues.

We can now define a $G$-spectrum.  Throughout this paper, $O (n)$
denotes the orthogonal group of $n\times n$ orthogonal real
matrices.  

\begin{definition}\label{defn-spectrum}
For a compact Lie group $G$, a \textbf{$G$-spectrum} $X$ is a sequence of
pointed $G\times O (n)$-spaces $X_{n}$ for $n\geq 0$, together with $G\times O
(p)\times O (q)$-equivariant structure maps 
\[
\nu_{p,q}\mathcolon S^{p} \wedge X_{q} \xrightarrow{} X_{p+q}
\]
that are associative and unital.  Here $S^{p}$ is the one-point
compactification of $\R^{p}$, so inherits an $G\times O (p)$-action
where $G$ acts trivially and the point at infinity is the fixed
basepoint.  The unital condition is simply that $\nu_{0,q}$ is the
identity.  The associative condition is that the composite
\[
S^{p} \wedge S^{q} \wedge X_{r} \xrightarrow{1\wedge \nu_{q,r}}
S^{p} \wedge X_{q+r} \xrightarrow{\nu_{p,q+r}} X_{p+q+r} 
\]
is equal to the composite 
\[
S^{p} \wedge S^{q} \wedge X_{r} \xrightarrow{\mu_{p,q}\wedge 1}
S^{p+q} \wedge X_{r} \xrightarrow{\nu_{p+q,r}} X_{p+q+r},
\]
where $\mu_{p,q}$ is the isomorphism induced by the standard
isomorphism $\R^{p} \oplus \R^{q}\cong \R^{p+q}$.  We will denote the
category of $G$-spectra by $\Gspectra$, where a map of $G$-spectra
$f\mathcolon X\xrightarrow{}Y$ is
a collection of $G\times O (n)$-equivariant maps $f_{n}\mathcolon
X_{n}\xrightarrow{}Y_{n}$ that are compatible with the structure
maps.  
\end{definition}

If $G=*$, a $G$-spectrum is just an orthogonal
spectrum~\cite{mandell-may-schwede-shipley}, and a $G$-spectrum for
general $G$ is just an orthogonal spectrum with an action of $G$ on
it.  As such, the category of $G$-spectra is closed symmetric
monoidal.  The easiest way to see this is to note that a $G$-spectrum
is an $S$-module in the category of $G$-orthogonal sequences, and $S$
is a commutative monoid in this closed symmetric monoidal category.
Here a \textbf{$G$-orthogonal sequence} $X$ is a collection of pointed
$G\times O (n)$-spaces $X_{n}$ for $n\geq 0$.  The category of
$G$-orthogonal sequences is closed symmetric monoidal where
\[
(X\otimes Y)_{n} = \bigvee_{p+q=n} O (n)_{+} \wedge_{O (p)\times O
(q)} (X_{p}\wedge Y_{q})
\]
with diagonal $G$-action.  The closed structure is given by 
\[
\hom (X,Y)_{n} = \prod_{m\geq n} \Map_{O (m-n)} (X_{m-n}, Y_{m})
\]
with $O (n)$ acting on a map by acting on the target $Y_{m}$ using the inclusion 
\[
O (n) \subseteq O (m-n)\times O (n) \xrightarrow{} O (m).
\]
Note that the maps in $\hom (X,Y)_{n}$ are not $G$-equivariant, so $G$
can act by conjugation as usual.  However, just as with $G$-spaces,
$G$-orthogonal sequences are enriched over (pointed) topological
spaces, and the enrichment is given by 
\[
\Map (X,Y) = \prod_{n} \Map_{G\times O (n)} (X_{n}, Y_{n}).  
\]

Now, $S$ is the $G$-orthogonal sequence whose $n$th space is $S^{n}$,
the one-point compactification of $\R^{n}$ with induced pointed
orthogonal action and trivial $G$-action.  This is a monoid using the
$G\times O (p)\times O (q)$-equivariant isomorphisms 
\[
S^{p}\wedge S^{q} \xrightarrow{} S^{p+q}.  
\]
It is a commutative monoid because the commutativity isomorphism of
the symmetric monoidal structure on $G$-orthogonal spectra involves a
$(p,q)$-shuffle, just as with symmetric
spectra~\cite{hovey-shipley-smith}.  The only thing this requires is
that the $(p,q)$-shuffle be an element of $O (p+q)$, which it of course
is.  

It is then clear that $G$-spectra are $S$-modules, and so inherit a
closed symmetric monoidal structure.  The category of $G$-spectra is
also enriched over topological spaces, where $\Map_{\Gspectra} (X,Y)$
is the subspace of $\Map (X,Y)$ consisting of maps of orthogonal
spectra.  

\section{Level structures}\label{sec-level}

A $G$-spectrum has levels $X_{n}$ for each integer $n$, and $X_{n}$ is
a $G\times O (n)$-space.  However, if $V$ is an orthogonal
representation of $G$ corresponding to $\rho \mathcolon
G\xrightarrow{}O (n)$, we can twist the $G$-action on $X_{n}$ by
$\rho$ to obtain a new $G$-space $X (V)$ with $X (V)=X_{n}$ as
an $O (n)$-space, but where
\[
g\cdot x = \rho (g) (gx) = g (\rho (g)x).  
\]
We can also think of 
\[
X (V) = O (\R^{n},V)_{+} \wedge_{O (n)} X_{n}
\]
with diagonal $G$-action, where $G$ acts on the set $O (\R^{n},V)$ of
orthogonal maps from $\R^{n}$ to $V$ by sending $\tau$ to $g\tau
g^{-1}$, which in this case is just $g\tau$ since $G$ acts trivially
on $\R^{n}$.

Note that $X (V)$ is not a $G\times O (n)$-space; it is instead an $O
(n)\rtimes_{\rho} G$-space, where the semi-direct product is taken
with respect to the action of $G$ on $O (n)$ where $g$ acts by
conjugating by $\rho (g)$.  

Let us denote by $\ev_{V}$ the evaluation functor $\ev_{V}\mathcolon
\Gspectra \xrightarrow{}G\Top_{*}$ that takes $X$ to $X (V)$.  This
functor should have a left adjoint $F_{V}$ whose $V$th space $(F_{V}K)
(V)$ is $O (V)_{+}\wedge K$; this means that if $n=\dim V$, we should
have
\[
(F_{V}K)_{n}= O (V, \R^{n})_{+} \wedge_{O (V)} (O (V)_{+}\wedge K) = O (V,
\R^{n})_{+}\wedge K.  
\]
In terms of the representation $\rho \mathcolon G\xrightarrow{}O (n)$
corresponding to $V$, we have 
\[
(F_{V}K)_{n} = O (n)_{+}\wedge K, 
\]
with $G$-action $g (\tau ,x)= (\tau \rho (g^{-1}),gx)$.  This
obviously commutes with $O (n)$-action, so gives us a $G\times O
(n)$-space.  

\begin{proposition}\label{prop-evaluation}
If $V$ is an orthogonal $n$-dimensional $G$-representation, the
functor $\ev_{V}\mathcolon \Gspectra \xrightarrow{}G\Top_{*}$ has a
left adjoint $F_{V}$ defined by
\[
(F_{V}K)_{n+k} = O (n+k)_{+} \wedge_{O (k)\times O (n)} (S^{k}\wedge
(O (V, \R^{n})_{+}\wedge K)),
\]
and there is a natural isomorphism 
\[
F_{V} (K) \wedge F_{W} (L) \cong F_{V\oplus W} (K\wedge L).  
\]
\end{proposition}

\begin{proof}
Let us first note that $F_{V}K$ is in the fact the free $S$-module on
the $G$-orthogonal sequence that is $O(V, \R^{n})_{+}\wedge K$ in
degree $n$ and the basepoint elsewhere.  Thus a map from $F_{V}K$ to a
spectrum $X$ is the same thing as a map of $G\times O (n)$-spaces 
\[
\alpha \mathcolon O (V, \R^{n})_{+}\wedge K \xrightarrow{} X_{n}
\]
The left-hand side is a free $O (n)$-space on $K$, but the $G$-action
is twisted.  Working this out gives that $f$ is equivalent to a map 
\[
\beta \mathcolon K\xrightarrow{}X_{n}
\]
such that $\beta (gk)=\rho (g) (g\beta (k))$, where $\rho \mathcolon
G\xrightarrow{}O(n)$ corresponds to the representation $V$.  This is
then the same thing as an equivariant map $K\xrightarrow{}\ev_{V}X$.  

For the last part of the proposition, because $F_{V} (K)$ is a free
$S$-module, it is enough to check that 
\begin{gather*}
O (n+m)_{+} \wedge_{O (n)\times O (m)} ((O (V,\R^{n})_{+}\wedge
K)\wedge (O (W, \R^{m})_{+}\wedge L)) \\
\cong O (V\oplus W, \R^{n+m})_{+} \wedge (K\wedge L)
\end{gather*}
as $G\times O (n+m)$-spaces.  We leave this to the reader.  
\end{proof}

Now by choosing a set of representations $V$, we can use the functors
$F_{V}$ and $\ev_{V}$ to construct a level model structure.  Note that
the functor $F_{0}$, where $0$ is the only $0$-dimensional
representation, plays a special role as it is symmetric monoidal.  

\begin{definition}\label{defn-level}
Given a set $\cat{U}$ of finite-dimensional orthogonal
$G$-representations, define a map $f$ of $G$-spectra to be
a $\cat{U}$-level equivalence (resp., $\cat{U}$-level fibration) if
$\ev_{V}f$ is a weak equivalence (resp. fibration) of based $G$-spaces
for all $V\in \cat{U}$.  Define $f$ to be a $\cat{U}$-cofibration if
it has the left lifting property with respect to all maps that are
both $\cat{U}$-level equivalences and $\cat{U}$-level fibrations.
\end{definition}

\begin{theorem}\label{thm-level}
The $\cat{U}$-cofibrations, $\cat{U}$-level fibrations, and
$\cat{U}$-level equivalences define a proper cellular topological
model structure on $G$-spectra.  This $\cat{U}$-level model structure
is symmetric monoidal when $\cat{U}$ is closed under finite direct
sums.
\end{theorem}

Of course, it is usual to take $\cat{U}$ to be a
\textbf{$G$-universe}; that is, a set of representations closed under
direct sums and summands that contains the one-dimensional trivial
representation.  At this point, it is unnecessary to put such a
restriction on $\cat{U}$.  

The proof of this theorem is standard, and so we only give a sketch
below.  
 
\begin{proof}
The generating cofibrations are the maps $F_{V}i$ for $V\in
\cat{U}$ and for $i$ a generating cofibration 
\[
(G/H \times S^{n-1})_{+}\xrightarrow{}(G/H \times D^{n})_{+}
\]
of based $G$-spaces.  The
generating trivial cofibrations are the maps $F_{V}j$ for $j$ a
generating trivial cofibration 
\[
(G/H \times D^{n})_{+} \xrightarrow{} (G/H\times D^{n}\times
D^{1})_{+}
\]
of based $G$-spaces.  The heart of the proof that this does define a
model structure is the fact that transfinite compositions of pushouts
of maps of the form $F_{V}j$ are $\cat{U}$-level equivalences.  The
basic point is that the maps $\ev_{W}F_{V}j$ are in fact inclusions of
$G$-deformation retracts, so that transfinite compositions of pushouts
of them are still $G$-homotopy equivalences and so weak equivalences.
In addition, one must also ensure that the set colimit is the same as
the space colimit, which can of course go wrong for weak Hausdorff
spaces.  This is dealt with just as in~\cite[Theorem
2.4]{mandell-may-equivariant-orthogonal}.

The key to proving that the $\cat{U}$-level model structure is
symmetric monoidal is the isomorphism $F_{V}i\boxprod F_{W}j\cong
F_{V\oplus W}(i\boxprod j)$, where $f\boxprod g$ is the map
\[
(\dom f\smash \codom g)\amalg_{\dom f\smash \dom g} (\codom f\smash \dom
g) \xrightarrow{} \codom f \smash \codom g.
\]
This isomorphism follows from the last part of
Proposition~\ref{prop-evaluation}, and makes it clear that the
$\cat{U}$-level model structure is symmetric monoidal
because $G\Top_{*}$ is so.  

The topological structure is similar but easier, since it is given by
the symmetric monoidal functor 
\[
\Top_{*} \xrightarrow{} G\Top_{*} \xrightarrow{F_{0}} \Gspectra 
\]
where the first map takes $X$ to $X$ with trivial $G$-action.  
\end{proof}

\section{The stable model structure}\label{sec-stable}

We now want to localize the $\cat{U}$-level model structure to produce
a stable model structure.  For any finite-dimensional orthogonal
representation $V$ of $G$, $S^{V}$ denotes the one-point
compactification of $V$ with fixed basepoint the point at infinity.
The point of the stable model structure is to make the $G$-spectra
$F_{0}S^{V}$ invertible under the smash product, for $V\in \cat{U}$,
so that we can desuspend by representation spheres in $\cat{U}$.  If
we want to get a symmetric monoidal result, we should assume that
$\cat{U}$ is closed under finite direct sums.  If we also want to get
a result that is stable in the usual sense of being able to desuspend
by the circle, we should assume that $\cat{U}$ contains the
one-dimensional trivial representation.  We will such a $\cat{U}$ a
\textbf{$G$-preuniverse}.  It is usual to assume that
$\cat{U}$ is closed under summands as well, so is a $G$-universe, but
this is not necessary.  We discuss this a bit more later.  

Now, the obvious candidate for an inverse of $S^{V}$ is $F_{V}S^{0}$,
because this is $S$ ``shifted by $V$.''  However, $F_{V}S^{0}\wedge
S^{V}$ is $F_{V} S^{V}$, when we want it to be $S$.  Fortunately,
there is a canonical map
\[
\lambda_{V}\mathcolon F_{V}S^{V} \xrightarrow{} S
\]
adjoint to the identity map 
\[
S^{V} \xrightarrow{}\ev_{V}S = O (\R^{n},V)_{+}\wedge_{O (n)} S^{n} \cong 
S^{V}
\]
where $n=\dim V$. 

Bousfield localization~\cite{hirschhorn} is a general theory that
starts with a (nice) model category and a map and produces a new model
category in which that map is now a weak equivalence, while
introducing as few other new weak equivalences as possible.  So we'd
like to define the $\cat{U}$-stable model structure as the (left)
symmetric monoidal Bousfield localization of the $\cat{U}$-level model
structure with respect to the maps $\lambda_{V}$ for $V$ an
irreducible representation in $\cat{U}$.  The second author has such a
theory of symmetric monoidal Bousfield localizations, but as it has
not appeared, we just carry it out in this special case, which is
simplified by the fact that $\lambda_{V}$ is a map between cofibrant
objects.  If $\lambda_{V}$ is to be a weak equivalence in a symmetric
monoidal model structure, we will need $\lambda_{V}\wedge A$ to be a
weak equivalence as well for all cofibrant $A$.  The cofibrant objects
in the level model structure are all built out of the domains and
codomains of the generating cofibrations of the level model structure.
In our case, the codomains of the generating cofibrations are
contractible, so we don't need them.

We therefore make the following definition. 

\begin{definition}\label{defn-stable}
For a compact Lie group $G$ and a $G$-preuniverse $\cat{U}$, we define
the \textbf{$\cat{U}$-stable model structure} on $G$-spectra to be the left
Bousfield localization of the $\cat{U}$-level model structure with
respect to the maps $\lambda_{V}\wedge F_{W} ((G/H)_{+}\wedge S^{n-1}_{+})$,
where $V,W\in \cat{U}$, $H$ is a closed subgroup of $G$, and $n\geq 0$.  
\end{definition}

Let us recall that Bousfield localization produces a new model
structure on the same category with the same cofibrations.  It works
by first constructing the locally fibrant objects and then using them
to construct the local equivalences.  In our case, then, a
$G$-spectrum $X$ will be \textbf{$\cat{U}$-stably fibrant} if it is
$\cat{U}$-level fibrant and the maps
\[
\Map (\lambda_{V}\wedge F_{W} ((G/H)_{+}\wedge S^{n-1}_{+}),X)
\]
are weak equivalences of topological spaces.  

Note that in Hirschhorn's book~\cite{hirschhorn} these mapping spaces
are in fact built from framings on the model category, and do not
refer to topological mapping spaces.  However, if the model category
is simplicial, the source is cofibrant, and the target is fibrant, the
mapping spaces created by framings are weakly equivalent to the
simplicial mapping spaces.  For simplicial model categories, then, we
can use simplicial mapping spaces instead of framings to form the
Bousfield localization with respect to $f$ if $f$ is a map of
cofibrant objects.  Every topological model category is also
simplicial through the geometric realization functor.  The simplicial
mapping spaces in a topological model category are just $\Sing \Map
(X,Y)$, where $\Sing $ denotes the singular complex functor.  But
$\Sing$ preserves and reflects weak equivalences.  Thus, for
topological model categories, we can use topological mapping spaces to
form the Bousfield localization with respect to $f$ as long as $f$ is
a map of cofibrant objects.  

The process of Bousfield localization then continues by defining a map
$f$ to be a \textbf{$\cat{U}$-stable equivalence} if $\Map (Qf,X)$ is
a weak equivalence of topological spaces for all $\cat{U}$-stably
fibrant $X$.  Here $Qf$ denotes any cofibrant approximation to $f$ in
the $\cat{U}$-level model structure.  That is, if $f\mathcolon
A\xrightarrow{}B$, then we would have a commutative square 
\[
\begin{CD}
A' @>Qf>> B' \\
@VVV @VVV \\
A @>>f> B
\end{CD}
\]
where $A'$ and $B'$ are cofibrant and the vertical maps are
$\cat{U}$-level equivalences.  Such an approximation is easily
obtained by taking a cofibrant approximation $QA$ to $A$ and then
factoring the composite $QA\xrightarrow{}A\xrightarrow{}B$ into a
cofibration followed by a trivial fibration in the $\cat{U}$-level
model structure.  The point of making this cofibrant approximation is
so we can use topological mapping spaces, as explained in the
preceding paragraph.  

The process of Bousfield localization concludes by defining $f$ to be
a \textbf{$\cat{U}$-stable fibration} if $f$ has the right lifting
property with respect to all maps that are both $\cat{U}$-level
cofibrations and $\cat{U}$-stable equivalences.  

Of course we expect the $\cat{U}$-stably fibrant objects to be
$\Omega$-spectra in an appropriate sense.  For this to make sense, we
need to note that any $G$-spectrum $X$ possesses natural maps 
\[
S^{V} \wedge X (W) \xrightarrow{} X (V\oplus W)
\]
that are both $G$ and $O (m)\times O (n)$-equivariant, where $m=\dim
V$ and $n=\dim W$.  Indeed, remember that $S^{V}=S^{m}$ as an $O
(m)$-space, and $X (W)=X_{n}$ as an $O (n)$-space, so these maps are
just the structure maps $\nu_{m,n}$ of $X$.  We just have to check
that $\nu_{m,n}$ is $G$-equivariant with respect to the twisted
$G$-actions.  So let $\rho_{1}\mathcolon G\xrightarrow{}O (m)$ and
$\rho_{2}\mathcolon G\xrightarrow{}O (n)$ denote the homomorphisms
corresponding to $V$ and $W$, so that the composite 
\[
\rho_{1}\times \rho_{2}\mathcolon G\xrightarrow{(\rho_{1},\rho_{2})} O
(m)\times O (n) \xrightarrow{} O (m+n)
\]
corresponds to $V\oplus W$.  We compute:
\begin{gather*}
\nu_{m,n} (g\cdot (x,y)) = \nu_{m,n} (\rho_{1} (g)gx,\rho_{2} (g)gy) \\
= (\rho_{1} (g)\times \rho_{2} (g))\nu_{m,n} (gx,gy) = (\rho_{1}\times
\rho_{2}) (g)g\nu_{m,n} (x,y),
\end{gather*}
as required.  

By taking adjoints, this means that any $G$-spectrum $X$ has maps 
\[
X (W) \xrightarrow{}\Omega^{V}X (V\oplus W) = \Map (S^{V}, X (V\oplus W))
\]
of $G$-spaces for all $V$ and $W$.  Note that $G$ acts by conjugation
on $\Omega^{V}X (V\oplus W)$, as usual with mapping spaces.  

\begin{definition}\label{defn-omega}
Given a $G$-preuniverse $\cat{U}$, a
\textbf{$\cat{U}-\Omega$-spectrum} is a $G$-spectrum $X$ such that the
map
\[
X (W) \xrightarrow{} \Omega^{V}X (V\oplus W)
\]
is a weak equivalence in $G\Top_{*}$ for all $V,W\in \cat{U}$.  
\end{definition}

\begin{theorem}\label{thm-omega}
The $\cat{U}$-stably fibrant $G$-spectra are the
$\cat{U}-\Omega$-spectra.  
\end{theorem}

\begin{proof}
We have a series of isomorphisms
\begin{gather*}
\Map_{\Gspectra} (F_{V}S^{V}\wedge F_{W} ((G/H)_{+}\wedge
S^{n-1}_{+}),X) \\
\cong \Map_{\Gspectra} (F_{V\oplus W} (S^{V}\wedge (G/H)_{+}\wedge
S^{n-1}_{+}),X) \\
\cong \Map_{G\Top_{*}} (S^{V}\wedge (G/H)_{+}\wedge S^{n-1}_{+}, X
(V\oplus W)) \\
\cong \Map_{G\Top_{*}} ((G/H)_{+}\wedge S^{n-1}_{+},\Omega^{V}X
(V\oplus W)) \\
\cong \Map_{\Top_{*}} (S^{n-1}_{+},(\Omega^{V}X (V\oplus W))^{H}).
\end{gather*}
and a similar isomorphism 
\[
\Map_{\Gspectra} (S\wedge F_{W} ((G/H)_{+}\wedge S^{n-1}_{+}),X)
\cong \Map_{\Top_{*}} (S^{n-1}_{+},(X (V))^{H}).
\]
Tracing the maps through this series of isomorphisms shows that $X$ is
$\cat{U}$-stably fibrant if and only if the map 
\[
X (V) \xrightarrow{}\Omega^{V}X (V\oplus W)
\]
is a weak equivalence of $G$-spaces for all $V$ and $W$ in $\cat{U}$.  
\end{proof}

\begin{corollary}\label{cor-lambda}
For $V\in \cat{U}$, the map $\lambda_{V}\mathcolon
F_{V}S^{V}\xrightarrow{}S$ is a $\cat{U}$-stable equivalence.  In fact
$\lambda_{V}\wedge F_{W}K$ is a $\cat{U}$-stable equivalence for all
$W\in \cat{V}$ and all cofibrant pointed $G$-spaces $K$.  
\end{corollary}

\begin{proof}
We get the first statement by taking $W=0$ in
Therorem~\ref{thm-omega}.  To get the general statement, we repeat the
argument of Theorem~\ref{thm-omega} to see that
\[
\Map_{\Gspectra } (\lambda_{V}\wedge F_{W}K,X) 
\]
is the map 
\[
\Map_{G\Top_{*}} (K, X (W)) \xrightarrow{} \Map_{G\Top_{*}} (K,\Omega^{V}X (V\oplus W)).
\]
If $X$ is a $\cat{U}$-stably fibrant $G$-spectrum, this map is a weak
equivalence since $K$ is cofibrant.  
\end{proof}

We then have the following theorem.  

\begin{theorem}\label{thm-stable}
Fix a $G$-preuniverse $\cat{U}$.  The category of $G$-spectra equipped
with the $\cat{U}$-stable model structure is a left proper cellular
topological stable symmetric monoidal model category in which the
$G$-spectra $F_{0}S^{V}$ for $V\in \cat{U}$ are invertible under the
smash product in the homotopy category.  
\end{theorem}

\begin{proof}
Bousfield localizations preserve left proper cellular model
categories.  To see that the $\cat{U}$-stable model structure is
symmetric monoidal, we start by showing that if $W\in \cat{U}$ and $K$
is a cofibrant pointed $G$-space, then $F_{W}K\wedge (-)$ is a left
Quillen functor with respect to the $\cat{U}$-stable model structure.
It is of course a left Quillen functor with respect to the
$\cat{U}$-level model structure.  The general theory of Bousfield
localization then tells us that it is a left Quillen functor with
respect to the $\cat{U}$-stable model structure if and only if
$F_{W}K\wedge f$ is a $\cat{U}$-stable equivalence for all the maps
$f=\lambda_{V}\wedge F_{W'} ((G/H)_{+}\wedge S^{n-1}_{+})$ with
respect to which we are localizing.  But $F_{W}K\wedge f$ is of the
form $\lambda_{V}\wedge F_{T} (L)$ for some $T\in \cat{U}$ and a
cofibrant $L$, so this follows from Corollary~\ref{cor-lambda}.  

Now, since the cofibrations don't change in passing to the stable
model structure, to prove that the $\cat{U}$-stable model structure is
symmetric monoidal, it suffices to check that $f\boxprod g$ is a
$\cat{U}$-stable equivalence when $f\mathcolon
F_{W}K\xrightarrow{}F_{W}L$ is one of the generating cofibrations of
the $\cat{U}$-level model structure and $g\mathcolon C\xrightarrow{}D$
is a cofibration and a $\cat{U}$-stable equivalence.  But we have just
seen that $F_{W}K\wedge (-)$ and $F_{W}L\wedge (-)$ are left Quillen
functors on the stable model category.  Therefore the map 
\[
F_{W}L\wedge C \xrightarrow{} (F_{W}L\wedge C)\amalg_{F_{W}K\wedge C}
(F_{W}K\wedge D)
\]
is a $\cat{U}$-stable trivial cofibration, as a pushout of
$F_{W}K\wedge g$.  But the composite
\[
F_{W}L\wedge C \xrightarrow{} (F_{W}L\wedge C)\amalg_{F_{W}K\wedge C}
(F_{W}K\wedge D) \xrightarrow{f\boxprod g} F_{W}L \wedge D
\]
is $F_{W}L\wedge g$, so it too is a $\cat{U}$-stable trivial
cofibration.  Thus $f\boxprod g$ must be a $\cat{U}$-stable
equivalence.  

The $\cat{U}$-stable model structure is topological through the same
left Quillen symmetric monoidal functor
$\Top_{*}\xrightarrow{}\Gspectra$ that takes $X$ to $F_{0}X$ with the
trivial $G$-action.  The fact that 
\[
\lambda_{V}\mathcolon F_{V}S^{V}\cong F_{V}S^{0}\wedge F_{0}S^{V}
\xrightarrow{}S
\]
is a $\cat{U}$-stable equivalence shows that $F_{V}S^{0}$ is a smash
inverse to $F_{0}S^{V}$ (for $V\in \cat{U}$) in the homotopy category of the
$\cat{U}$-stable model structure.  In particular, we can take $V$ to
to be the one-dimensional trivial representation to see that the
suspension is invertible in the homotopy category, so the
$\cat{U}$-stable model structure is in fact stable in the usual
sense.  
\end{proof}

It is natural to think that if $\cat{U}$ is a $G$-preuniverse and
$\cat{U}'$ is the $G$-universe generated by $\cat{U}$, so just the
collection of summands of $\cat{U}$, then the $\cat{U}$-stable model
structure should be equivalent to the $\cat{U}'$-stable model
structure.  The argument for this would be that if $V\oplus W$ is in
$\cat{U}$, then the map
\[
S^{V} \wedge (S^{W}\wedge F_{V\oplus W}S^{0}) =S^{V\oplus W}\wedge
F_{V\oplus W}S^{0} \xrightarrow{} S
\]
is a weak equivalence in $\cat{U}$-stable model structure, and so
$S^{W}\wedge F_{V\oplus W}S^{0}$ is a smash inverse of $S^{V}$.  This
is wrong, though, because the left-hand side is not the derived smash
product since neither factor is cofibrant in the $\cat{U}$-model
structure.  So we cannot say that $S^{V}$ is invertible under smash
product in the homotopy category of the $\cat{U}$-stable model
structure.

\section{Comparison to Mandell-May method}\label{sec-comparison}

In this section, we compare our approach to equivariant stable
homotopy theory to the approach of Mandell and
May~\cite{mandell-may-equivariant-orthogonal}.  If we fix a universe
$\cat{U}$, a $G\cat{U}$-spectrum is in particular a $G$-functor from the
universe, thought of as a $G$-category via $\cat{U} (V,W)=O (V,W)_{+}$
with diagonal $G$-action when $\dim V=\dim W$ and $*$ otherwise, to
the category of pointed $G$-spaces. That is, such a functor has
natural $G$-equivariant maps 
\[
O (V,W)_{+} \wedge X (V) \xrightarrow{}X (W)
\]
that are associative and unital.  A $G\cat{U}$-spectrum is such a
functor equipped with an associative and unital natural transformation 
\[
S^{V} \wedge X (W) \xrightarrow{} X (V\oplus W)
\]
of $G$-functors on $\cat{U}\times \cat{U}$.  

There is then an obvious forgetful functor $\beta $ from
$G\cat{U}$-spectra to $G$-spectra with $(\beta X)_{n}=X (\R^{n})$.
Note that $(\beta X)_{n}$ is an $G\times O (n)$-space through the
$G$-map 
\[
O (\R^{n},\R^{n})_{+}\wedge X (\R^{n}) \xrightarrow{} X (\R^{n}).
\]
We get $G$-maps 
\[
S^{p} \wedge (\beta X)_{q} \xrightarrow{} (\beta X)_{p+q}
\]
by restricting the structure map of $X$ to $V=\R^{p}$.  These maps are
$G\times O (p)\times O (q)$-equivariant because the structure map of
$X$ is a natural transformation of $G$-functors on $\cat{U}\times
\cat{U}$.  They are associative and unital because the structure map
of $X$ is so.  

Conversely, we define a functor $\alpha$ from $G$-spectra to
$G\cat{U}$-spectra by defining 
\[
(\alpha X) (V) = O (\R^{n},V)_{+} \wedge_{O (n)} X_{n}
\]
where $n=\dim V$.  Equivalently, we define $(\alpha X) (V)=X_{n}$ with
group action 
\[
g\cdot x = \rho (g) (gx) = g (\rho (g)x)
\]
where $\rho\mathcolon G\xrightarrow{}O (n)$ is the representation
corresponding to $V$.  Then $\alpha X$ becomes a $G$-functor from
$\cat{U}$ to pointed $G$-spaces, because the map
\[
j_{V,W}\mathcolon O (V,W)_{+} \wedge X (V) \xrightarrow{} X (W)
\]
defined by $j (\tau ,x)=\tau x$ for $\tau \in O (V,W)=O (n)$ and $x\in
X (V)=X_{n}=X (W)$ is $G$-equivariant.  Indeed, let
$\rho_{1},\rho_{2}\mathcolon G\xrightarrow{}O (n)$ correspond to $V$
and $W$, respectively.  Then we compute:
\begin{gather*}
j_{V,W} (g (\tau ,x)) = j_{V,W} (\rho_{2} (g)\tau \rho_{1}
(g)^{-1},\rho_{1} (g) (gx)) \\
= \rho_{2} (g)
\tau \rho_{1} (g)^{-1} \rho_{1} (g) (gx) \\
= \rho_{2} (g)\tau (gx)
=\rho_{2} (g)g (\tau x) = g\cdot j_{V,W} (\tau ,x).  
\end{gather*}
A similar computation shows that the structure maps 
\[
\nu_{p.q}\mathcolon S^{p}\wedge X_{q} \xrightarrow{} X_{p+q}
\]
are $G$-equivariant maps 
\[
S^{V} \wedge X (W) \xrightarrow{} X (V\oplus W).
\]
In fact, we have already done this, just before
Definition~\ref{defn-omega}.  We leave the proof that the structure
maps are compatible with the $O (V,V')$ and $O (W,W')$-actions, so
define a natural transformation of $G$-functors on $\cat{U}\times
\cat{U}$, to the reader.  The associativity and unit axioms for $\alpha
X$ follow immediately from the ones for $X$, since the structure maps
are the same.  

Note that the composite functor $\beta \alpha$ is the identity
functor.  On the other hand, if $X$ is a $G\cat{U}$-spectrum and $\dim V=n$, then the
$G$-map 
\[
X (\R^{n},V)_{+} \wedge X (\R^{n}) \xrightarrow{} X (V),
\]
coming from the fact that $X$ is a $G$-functor,
descends to an isomorphism 
\[
X (\R^{n},V)_{+} \wedge_{O (n)} X (\R^{n}) \xrightarrow{} X (V)
\]
and so an isomorphism $\alpha \beta X\xrightarrow{}X$.  

We have therefore proved the following proposition, also proved
in~\cite[Theorem~V.1.5]{mandell-may-equivariant-orthogonal}.

\begin{proposition}\label{prop-equiv}
The functors $\alpha$ and $\beta$ are adjoint equivalences of
categories.  
\end{proposition}

Mandell and May also show that both $\alpha$ and $\beta$ are symmetric
monoidal, and of course they commute with the functors $F_{V}$ and
$\ev_{V}$ for $V\in \cat{U}$.  

The following proposition is then clear. 

\begin{proposition}\label{prop-equiv-level}
Let $\cat{U}$ be a $G$-universe.  With respect to the adjoint
equivalences $\alpha$ and $\beta$, the $\cat{U}$-level model structure
on $G$-spectra and the level model structure on $G\cat{U}$-spectra
coincide.  That is, $\alpha$ and $\beta$ preserve and reflect
cofibrations, fibrations, and weak equivalences.
\end{proposition}

Of course, we want the stable model structures to coincide as well,
and they do.  

\begin{theorem}\label{thm-equiv}
Let $\cat{U}$ be a $G$-universe.  With respect to the adjoint
equivalences $\alpha$ and $\beta$, the $\cat{U}$-stable model
structure on $G$-spectra and the stable model structure on
$G\cat{U}$-spectra coincide.  
\end{theorem}

\begin{proof}
We first prove that $\alpha$ is a left Quillen functor.  Since
$\alpha$ is already a left Quillen functor on the $\cat{U}$-level
model structure, the general theory of Bousfield localization tells us
that $\alpha$ is a left Quillen functor on the $\cat{U}$-stable model
structure if and only if $\alpha (\lambda_{V}\wedge F_{W}
((G/H)_{+}\wedge S^{n-1}_{+}))$ is a stable equivalence for all
$V,W\in \cat{U}$, closed subgroups $H$ and $n\geq 0$.  But $\alpha$ is
symmetric monoidal, so this is 
\[
\alpha (\lambda_{V}) \wedge \alpha (F_{W} ((G/H)_{+}\wedge S^{n-1}_{+})).
\]
The map $\alpha (\lambda_{V})$ is proved to be a a stable equivalence
(that is, an isomorphism on stable homotopy groups)
in~\cite[Lemma~III.4.5]{mandell-may-equivariant-orthogonal}.  Since it
is a stable equivalence of cofibrant objects in a symmetric monoidal
model category, it remains so after smashing with any cofibrant
object, such as $F_{W} ((G/H)_{+}\wedge S^{n-1}_{+})$.  Thus $\alpha$
is a left Quillen functor.  

It now follows that $\alpha$ preserves all stable equivalences.
Indeed, if $f\mathcolon A\xrightarrow{}B$ is a stable equivalence, we
can take a cofibrant approximation $Qf$ to $f$ that is level
equivalent to $f$.  More precisely, we take a level equivalence
$p\mathcolon QA\xrightarrow{}A$ where $QA$ is cofibrant, and then
factor $f\circ p$ into a cofibration $Qf\mathcolon QA\xrightarrow{}QB$
followed by a level equivalence $q\mathcolon QB\xrightarrow{}B$.  Then
$Qf$ is a stable equivalence between cofibrant objects.  Thus $\alpha
(Qf)$ is a stable equivalence.  Since $\alpha (p)$ and $\alpha (q)$
are level equivalences, it follows that $\alpha (f)$ is a stable
equivalence.

Of course, $\alpha$ and $\beta$ also preserve and reflect stably
fibrant objects, since these are $\Omega$-spectra with respect to
$\cat{U}$ in both cases. (These are called $\Omega -G$-spectra by
Mandell and May, and~ Corollary~III.4.10
of~\cite{mandell-may-equivariant-orthogonal} identifies them as the
stably fibrant objects).  We now use this to show that $\beta$
preserves stable equivalences whose target is stably fibrant.  Indeed,
suppose $f\mathcolon X\xrightarrow{}Y$ is a stable equivalence of
$G\cat{U}$-spectra and $Y$ is stably fibrant.  Let $j\mathcolon \beta
X\xrightarrow{}Z$ be a stable trivial cofibration to a stably fibrant
$G$-spectrum $Z$.  Then
\[
\alpha j\mathcolon X\cong \alpha (\beta X) \xrightarrow{}\alpha Z
\]
is a stable trivial cofibration to the stably fibrant
$G\cat{U}$-spectrum $\alpha Z$.  Since $Y$ is stably fibrant, we can
find a lift $g\mathcolon \alpha (Z)\xrightarrow{}Y$ in the commutative
diagram 
\[
\begin{CD}
X @>f>> Y \\
@V\alpha j VV @VVV \\
\alpha Z @>>> *
\end{CD}
\]
such that $g\circ (\alpha j)=f$.  Thus $g$ is a stable equivalence
between stably fibrant $G\cat{U}$-spectra.  By Theorem~V.3.4
of~\cite{mandell-may-equivariant-orthogonal}, $g$ is a level
equivalence.  Thus $\beta g$ is a $\cat{U}$-level equivalence, and so
\[
\beta f = (\beta g)\circ (\beta \alpha j) = (\beta g)\circ j
\]
is a $\cat{U}$-stable equivalence.  

We can now use this to prove that $\beta$ preserves arbitrary stable
equivalences.  Indeed, if $f\mathcolon X\xrightarrow{}Y$ is a stable
equivalence, let $j\mathcolon Y\xrightarrow{}RY$ be a stable trivial
cofibration to a stably fibrant $G\cat{U}$-spectrum $RY$.  Then factor
$j\circ f= (Rf)\circ i$, where $i\mathcolon X\xrightarrow{}RX$ is a
stable trivial cofibration and $Rf\mathcolon RX\xrightarrow{}RY$ is a
stable fibration.  Note that $Rf$ is necessarily a stable
equivalence.  Applying $\beta$, and using the fact that $\beta$
preserves stable equivalences whose target is stably fibrant, we see
that $\beta i$, $\beta j$, and $\beta (Rf)$ are all stable
equivalences.  It follows that $\beta f$ is also a stable
equivalence.  

Since $\beta$ preserves cofibrations and stable equivalences and has
right adjoint $\alpha$, $\beta$ is a left Quillen functor with respect
to the stable model structures.  It follows that $\alpha$ preserves
fibrations.  Of course, $\alpha$ is also a left Quillen functor, so
$\beta$ preserves fibrations as well, completing the proof.  
\end{proof}

Of course, since the $\cat{U}$-stable model structure coincides with
the Mandell-May model structure under the equivalences $\alpha$ and
$\beta$, all of the properties that Mandell and May prove hold for the
$\cat{U}$-stable model structure as well.  

We therefore have the following corollary. 

\begin{corollary}\label{cor-homotopy}
Suppose $\cat{U}$ is a $G$-universe. 
\begin{enumerate}
\item The weak equivalences in the $\cat{U}$-stable model structure
are the maps that induce isomorphisms on all stable homotopy groups
$\pi_{q}^{H} (-)$ for $q$ an integer and $H$ a closed subgroup of $G$,
where
\[
\pi_{q}^{H} (X) = \colim_{V\in \cat{U}} \pi_{q} (\Omega^{V}X (V)^{H})
\]
for $q\geq 0$ and
\[
\pi_{q}^{H} (X) = \colim_{V\in \cat{U}} \pi_{0}^{H} (\Omega^{V}X (V\oplus \R^{q})^{H})
\]
for $q<0$.  
\item The stable fibrations in the $\cat{U}$-stable model structure
are the level fibrations $p\mathcolon X\xrightarrow{}Y$ such that the
diagram of $G$-spaces 
\[
\begin{CD}
X (V) @>>> \Omega^{W} X (V\oplus W) \\
@VpVV @VV\Omega^{W}\!p V \\
Y (V) @>>> \Omega^{W}Y (V\oplus W)
\end{CD}
\]
is an homotopy pullback for all $V, W\in \cat{U}$.  
\item The $\cat{U}$-stable model structure is right proper.
\item The $\cat{U}$-stable model structure satisfies the monoid
axiom.  
\item Cofibrant objects are flat in the $\cat{U}$-stable model
structure, in the sense that if $X$ is cofibrant then $X\wedge (-)$
preserves stable equivalences.  
\end{enumerate}
\end{corollary}

The last two properties are the essential properties needed for a good
theory of monoids and modules over them in a model category.  For a
good theory of commutative monoids, one typically needs a positive
model structure, and this too is provided in Mandell and May, and so
holds also for the $\cat{U}$-stable model structure.  We mention that
Stolz~\cite{stolz-thesis} also has a different positive model
structure, analogous to the convenient positive model structure of
Shipley~\cite{shipley-convenient}, where a cofibration of commutative
monoids is in particular a cofibration in the underlying category with
its positive model structure.  

\providecommand{\bysame}{\leavevmode\hbox to3em{\hrulefill}\thinspace}
\providecommand{\MR}{\relax\ifhmode\unskip\space\fi MR }
\providecommand{\MRhref}[2]{%
  \href{http://www.ams.org/mathscinet-getitem?mr=#1}{#2}
}
\providecommand{\href}[2]{#2}


\begin{thebibliography}{LMSM86}

\bibitem[EM97]{elmendorf-may}
A.~D. Elmendorf and J.~P. May, \emph{Algebras over equivariant sphere spectra},
  J. Pure Appl. Algebra \textbf{116} (1997), no.~1-3, 139--149, Special volume
  on the occasion of the 60th birthday of Professor Peter J. Freyd. \MR{1437616
  (98b:55009)}

\bibitem[Fau08]{fausk}
Halvard Fausk, \emph{Equivariant homotopy theory for pro-spectra}, Geom. Topol.
  \textbf{12} (2008), no.~1, 103--176. \MR{2377247 (2009c:55010)}

\bibitem[Hir03]{hirschhorn}
Philip~S. Hirschhorn, \emph{Model categories and their localizations},
  Mathematical Surveys and Monographs, vol.~99, American Mathematical Society,
  Providence, RI, 2003. \MR{2003j:18018}

\bibitem[Hov99]{hovey-model}
Mark Hovey, \emph{Model categories}, American Mathematical Society, Providence,
  RI, 1999. \MR{99h:55031}

\bibitem[HSS00]{hovey-shipley-smith}
Mark Hovey, Brooke Shipley, and Jeff Smith, \emph{Symmetric spectra}, J. Amer.
  Math. Soc. \textbf{13} (2000), no.~1, 149--208. \MR{2000h:55016}

\bibitem[LMSM86]{lewis-may-steinberger}
L.~G. Lewis, Jr., J.~P. May, M.~Steinberger, and J.~E. McClure,
  \emph{Equivariant stable homotopy theory}, Springer-Verlag, Berlin, 1986,
  With contributions by J. E. McClure. \MR{88e:55002}

\bibitem[MM02]{mandell-may-equivariant-orthogonal}
M.~A. Mandell and J.~P. May, \emph{Equivariant orthogonal spectra and
  {$S$}-modules}, Mem. Amer. Math. Soc. \textbf{159} (2002), no.~755, x+108.
  \MR{2003i:55012}

\bibitem[MMSS01]{mandell-may-schwede-shipley}
M.~A. Mandell, J.~P. May, S.~Schwede, and B.~Shipley, \emph{Model categories of
  diagram spectra}, Proc. London Math. Soc. (3) \textbf{82} (2001), no.~2,
  441--512. \MR{2001k:55025}

\bibitem[Shi04]{shipley-convenient}
Brooke Shipley, \emph{A convenient model category for commutative ring
  spectra}, Homotopy theory: relations with algebraic geometry, group
  cohomology, and algebraic {$K$}-theory, Contemp. Math., vol. 346, Amer. Math.
  Soc., Providence, RI, 2004, pp.~473--483. \MR{2066511 (2005d:55014)}

\bibitem[Sto11]{stolz-thesis}
Martin Stolz, \emph{Equivariant structure on smash powers of commutative ring
  spectra}, Ph.D. thesis, University of Bergen, 2011.

\end{thebibliography}

\end{document}